\documentclass[a4paper]{article}
\usepackage{amsmath,amsthm,amssymb}
\usepackage{graphicx}
\usepackage{multirow}
\usepackage{dcolumn}
\usepackage{cases}
\newtheorem{Def}{Definition}[section]
\newtheorem{theorem}[Def]{Theorem}
\newtheorem{lemma}[Def]{Lemma}

\newtheorem{cor}[Def]{Corollary}

\makeatletter
 
\newcolumntype{d}[1]{D{.}{\cdot}{#1}}
\renewcommand{\@cite}[2]{{#1\if@tempswa , #2\fi}}
\makeatother

\title{On upper bounds on stable commutator lengths in mapping class groups}
\author{Naoyuki Monden}
\date{}
\begin{document}
\maketitle

\begin{abstract}
We give new upper bounds on the stable commutator lengths of Dehn twists in mapping class groups and new lower bounds on the stable commutator lengths of Dehn twists in hyperelliptic mapping class groups.
In particular, we show that the stable commutator lengths of Dehn twists about a nonseparating and a separating curve on an oriented closed surface of genus 2 are not equal to each other.
\end{abstract}

\section{Introduction}
Let $G$ be a group, and let $[G,G]$ denote the commutator subgroup of $G$.
Given $x\in[G,G]$ the \textit{commutator length} ${\rm cl}_{G}(x)$ of $x$ is the least number of commutators in $G$ whose product is equal to $x$.
The \textit{stable commutator length} ${\rm scl}_{G}(x)$ is the limit of ${\rm cl}_{G}(x^{n})/n$ as $n$ goes to infinity.
If $x^{m}\in [G,G]$ for some positive integer $m$, define ${\rm scl}_{G}(x)={\rm scl}_{G}(x^{m})/m$, and define ${\rm scl}_{G}(x)=\infty$ if no power of $x$ is contained in $[G,G]$
(We refer the reader to [\cite{14}] for the details of the theory of the stable commutator length).

Let $\Sigma_{g}$ be a closed connected oriented surface of genus $g\geq 2$ embedded in $\mathbb{R}^{3}$ in Figure~\ref{fig4}.
We can define the hyperelliptic involution $\iota :\Sigma_{g}\rightarrow \Sigma_{g}$ as in Figure~\ref{fig4}.
\begin{figure}[htbp]
 \begin{center}
  \includegraphics*[width=7cm]{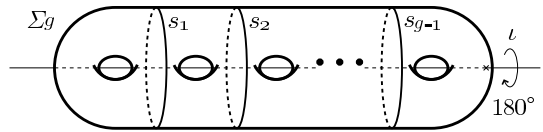}
 \end{center}
 \caption{hyperelliptic involution $\iota$ and the curves $s_{1},\ldots s_{g-1}$}
 \label{fig4}
\end{figure}

Let ${\mathcal M}_{g}$ be the mapping class group of $\Sigma_{g}$, i.e., the group of isotopy classes of orientation-preserving diffeomorphisms of $\Sigma_{g}$.
Let $t_{c}$ and $t_{s}$ denote the right handed Dehn twist about a nonseparating curve $c$ and a nontrivial separating curve $s$ on $\Sigma_{g}$, respectively.

It is well known that ${\mathcal M}_{g}/[{\mathcal M}_{g},{\mathcal M}_{g}]$ is generated by the class of a Dehn twist about a nonseparating simple closed curve and is equal to $\mathbb{Z}_{10}$ if $g=2$ and trivial if $g\geq 3$.
Therefore, If $g\geq 3$, then $t_{c}$ and $t_{s}$ are in $[{\mathcal M}_{g},{\mathcal M}_{g}]$.
Moreover, it is well known that $t_{c}^{10}$ and $t_{s}^{5}$ are in $[{\mathcal M}_{2},{\mathcal M}_{2}]$.
Hence, for any $g\geq 2$ we can define ${\rm scl}_{{\mathcal M}_{g}}(t_{c})$ and ${\rm scl}_{{\mathcal M}_{g}}(t_{s})$.

Endo and Kotschick proved that ${\rm scl}_{{\mathcal M}_{g}}(t_{s})\geq 1/(18g-6)$ (see [\cite{2}]).
Consequently, they proved that for any $g\geq 2$ ${\mathcal M}_{g}$ is not uniformly perfect and that the natural map from the second bounded cohomology to the ordinary cohomology of ${\mathcal M}_{g}$ is not injective, which verified two conjectures of Morita [\cite{12}].
Korkmaz proved that ${\rm scl}_{{\mathcal M}_{g}}(t_{c})\geq 1/(18g-6)$ (see [\cite{9}]).
He also gave upper bounds for the stable commutator lengths of Dehn twists.
He showed that (1) ${\rm scl}_{{\mathcal M}_{g}}(t_{c})\leq 3/20$ for $g\geq 2$ and (2) ${\rm scl}_{{\mathcal M}_{g}}(t_{s})\leq 3/4$ for $g\geq 3$ by using results of [\cite{1}] and [\cite{3}] and by showing that $t_{c}^{10}$ is written as a product of two commutators.
By using the results of quasimorphisms we give the following upper bounds.
\begin{theorem}\label{thm1}
Let $c$ and $s$ be a nonseparating curve and a separating curve on $\Sigma_{g}$ $(g\geq 2)$, respectively.
\begin{description}
\item[{\rm (1)}] If $g\geq 2$, then ${\rm scl}_{{\mathcal M}_{g}}(t_{c})\leq \displaystyle\frac{1}{10}$,
\item[{\rm (2)}] If $g\geq 3$, then ${\rm scl}_{{\mathcal M}_{g}}(t_{s})\leq \displaystyle\frac{1}{2}$,
\item[{\rm (3)}] If $g=2$, then ${\rm scl}_{{\mathcal M}_{2}}(t_{s})\leq \displaystyle\frac{7}{10}$.
\end{description}
\end{theorem}

Let ${\mathcal H}_{g}$ be the hyperelliptic mapping class group of genus $g$, i.e., the subgroup of ${\mathcal M}_{g}$ which consists of all isotopy classes of orientation-preserving diffeomorphisms of $\Sigma_{g}$ commuting with the isotopy class of $\iota$.
Let $s_{1},\ldots,s_{g-1}$ be separating curves as shown in Figure~\ref{fig4}.
Each $t_{s_{h}}$ $(h=1,\ldots, g)$ belongs to ${\mathcal H}_{g}$ and $t_{s_{g-h}}$ is conjugate to $t_{s_{h}}$ in ${\mathcal H}_{g}$.
Since the stable commutator length is constant on conjugacy classes, it suffices to consider $t_{s_{h}}$ $(h=1,\ldots, [\frac{g}{2}])$.
It is well known that $t_{c}^{4(2g+1)}$ and $t_{s_{h}}^{4(2g+1)}$ are in $[{\mathcal H}_{g},{\mathcal H}_{g}]$.
Hence, we can define ${\rm scl}_{{\mathcal H}_{g}}(t_{c})$ and ${\rm scl}_{{\mathcal H}_{g}}(t_{s_{h}})$.

Endo and Kotschick proved that ${\rm scl}_{{\mathcal H}_{g}}(t_{s_{h}})\geq 1/(18g-6)$ (see [\cite{2}]).
We give the following lower bounds on the stable commutator lengths of Dehn twists in ${\mathcal H}_{g}$.
\begin{theorem}\label{thm2}
For all $g\geq 2$,
\begin{center}
{\rm (1)} \ $\displaystyle\frac{1}{4(2g+1)}\leq{\rm scl}_{{\mathcal H}_{g}}(t_{c})$, \ \ {\rm (2)} \ $\displaystyle\frac{h(g-h)}{g(2g+1)}\leq{\rm scl}_{{\mathcal H}_{g}}(t_{s_{h}})$ \ \ $(h=1,\ldots,[\frac{g}{2}])$.
\end{center}
\end{theorem}

In particular, Since ${\mathcal M}_{2}$ is equal to ${\mathcal H}_{2}$, by combining Theorem~\ref{thm1} with Theorem~\ref{thm2} we have $\frac{1}{20}\leq{\rm scl}_{{\mathcal M}_{2}}(t_{c})\leq \frac{1}{10}\leq{\rm scl}_{{\mathcal M}_{2}}(t_{s})\leq \frac{7}{10}$.
On the other hand, we give the following lemma.
\begin{lemma}\label{10}
$\displaystyle 6{\rm scl}_{{\mathcal M}_{2}}(t_{c})\leq \frac{1}{2}{\rm scl}_{{\mathcal M}_{2}}(t_{s})+\frac{1}{2}$.
\end{lemma}
If ${\rm scl}_{{\mathcal M}_{2}}(t_{c})={\rm scl}_{{\mathcal M}_{2}}(t_{s})$, then by Lemma~\ref{10} we have ${\rm scl}_{{\mathcal M}_{2}}(t_{s})\leq \frac{1}{11}$.
This contradicts to $\frac{1}{10}\leq{\rm scl}_{{\mathcal M}_{2}}(t_{s})$.
Therefore, we have the following results.
\begin{cor}
If $g=2$, then
\begin{description}
\item[{\rm (1)}] $\displaystyle\frac{1}{20}\leq{\rm scl}_{{\mathcal M}_{2}}(t_{c})\leq \frac{1}{10}\leq {\rm scl}_{{\mathcal M}_{2}}(t_{s})\leq \frac{7}{10}$.
\item[{\rm (2)}] ${\rm scl}_{{\mathcal M}_{2}}(t_{c})$ is not equal to ${\rm scl}_{{\mathcal M}_{2}}(t_{s})$.
\end{description}
\end{cor}

\section{Preliminaries}
\subsection{Stable commutator lengths and quasimorphisms}
Let $G$ denote a group and let $[G,G]$ denote the commutator subgroup, the subgroup of $G$ generated by all commutators $[x,y]=xyx^{-1}y^{-1}$ for $x,y\in G$.
\begin{Def}
For $x\in [G,G]$, the commutator length ${\rm cl}_{G}(x)$ of $x$ is the least number of commutators in $G$ whose product is equal to $x$.
\end{Def}
\begin{Def}
For $x\in[G,G]$,
\begin{eqnarray*}
{\rm scl}_{G}(x)=\lim_{n\rightarrow \infty}\frac{{\rm cl}_{G}(x^{n})}{n}
\end{eqnarray*}
is called the stable commutator length of $x$.

For each fixed $x$, the function $n\rightarrow {\rm cl}_{G}(x^{n})$ is non-negative and ${\rm cl}_{G}(x^{m+n})\leq {\rm cl}_{G}(x^{m})+{\rm cl}_{G}(x^{n})$; hence this limit exists.
If $x$ is not in $[G,G]$ but has a power $x^{m}$ which is, define ${\rm scl}_{G}(x) = {\rm scl}_{G}(x^{m})/m$.
We define ${\rm scl}_{G}(x)=\infty$ if no power of $x$ is contained in $[G,G]$.
\end{Def}
\begin{Def}
A quasimorphism is a function
\begin{eqnarray*}
\phi:G\rightarrow \mathbb{R}
\end{eqnarray*}
for which there is a least constant $D(\phi)\geq 0$ such that
\begin{eqnarray*}
|\phi(xy)-\phi(x)-\phi(y)|\leq D(\phi)
\end{eqnarray*}
for all $x,y\in G$.
We call $D(\phi)$ the defect of $\phi$.
\end{Def}
\begin{Def}
A quasimorphism is homogeneous if it satisfies the additional property
\begin{eqnarray*}
\phi(x^{n})=n\phi({x})
\end{eqnarray*}
for all $x\in G$ and $n\in \mathbb{Z}$.
\end{Def}
We recall the following basic facts.
\begin{lemma}\label{quasi1}
Let $\phi$ be a homogeneous quasimorphism.
For all $x,y\in G$,
\begin{description}
\item{{\rm (a)}} $\phi(x)=\phi(yxy^{-1})$,
\item{{\rm (b)}} $xy=yx\Rightarrow \phi(xy)=\phi(x)+\phi(y)$.
\end{description}
\end{lemma}
\begin{proof}
For any positive integer $n$,
\begin{eqnarray*}
|\phi(yxy^{-1})-\phi(x)|=\frac{1}{n}|\phi(yx^{n}y^{-1})-\phi(x^{n})|\leq \frac{2D(\phi)}{n}.
\end{eqnarray*}
Hence, we have $\phi(x)=\phi(yxy^{-1})$.

Suppose that $xy=yx$.
For any positive integer $n$,
\begin{eqnarray*}
|\phi(xy)-\phi(x)-\phi(y)|&=&\frac{1}{n}|\phi((xy)^{n})-\phi(x^{n})-\phi(y^{n})|\\
&=&\frac{1}{n}|\phi(x^{n}y^{n})-\phi(x^{n})-\phi(y^{n})|\leq \frac{1}{n}D(\phi).
\end{eqnarray*}
Hence, $\phi(xy)=\phi(x)+\phi(y)$.

\end{proof}
\begin{theorem}[Bavard's Duality Theorem {\rm [\cite{1}]}]
Let $Q$ be the set of homogeneous quasimorphisms on $G$.
For any $x\in [G,G]$, we have
\begin{eqnarray*}
{\rm scl}_{G}(x)=\displaystyle\sup_{\phi\in Q}\frac{|\phi(x)|}{2D(\phi)}.
\end{eqnarray*}
\end{theorem}

\subsection{Relations in mapping class groups}
Hereafter, we do not distinguish a simple closed curve and its isotopy class.
The next lemmas are well known.
\begin{lemma}\label{conj}
For any $f\in {\mathcal M}_{g}$ and any simple closed curve $c$ in $\Sigma_{g}$ we have
\begin{eqnarray*}
t_{f(c)}=ft_{c}f^{-1}.
\end{eqnarray*}
\end{lemma}
\begin{lemma}\label{braid}
Let $c$ and $d$ be two simple closed curves on $\Sigma_{g}$.

{\rm (a)} If $c$ is disjoint from $d$, then $t_{c}t_{d}=t_{d}t_{c}$.

{\rm (b)} If $c$ intersects $d$ in one point transversely, then $t_{c}t_{d}t_{c}=t_{d}t_{c}t_{d}$.
\end{lemma}
The following two relations in ${\mathcal M}_{g}$ are also well known.
The first one is the \textit{lantern relation}.
This relation was discovered by Dehn (see [\cite{13}]) and was rediscovered by Johnson (see [\cite{4}]).
Let $s$, $a$, $b$, $c$, $x$, $y$ and $z$ be simple closed curves on $\Sigma_{g}$ $(g\geq 3)$ in Figure~\ref{fig1}.
The Dehn twists about these simple closed curves satisfy the \textit{lantern relation}
\begin{eqnarray*}
t_{s}t_{a}t_{b}t_{c}=t_{x}t_{y}t_{z}.
\end{eqnarray*}
\begin{figure}[htbp]
 \begin{center}
  \includegraphics*[width=6cm]{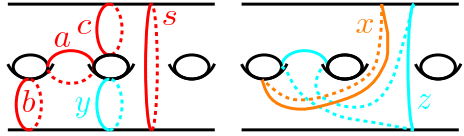}
 \end{center}
 \caption{The curves $s$, $a$, $b$, $c$, $x$, $y$ and $z$ on $\Sigma_{g}$ $(g\geq 3)$.}
 \label{fig1}
\end{figure}

In the case of $g=2$, we define the curves $s$, $a_{1}$, $a_{3}$, $a_{5}$ and $x$ on $\Sigma_{2}$ in Figure~\ref{fig3}.
$t_{s}$, $t_{a_{1}}$, $t_{a_{3}}$, $t_{a_{5}}$ and $t_{x}$ satisfy the \textit{lantern relation}
\begin{eqnarray*}
t_{a_{1}}^{2}t_{a_{5}}^{2}=t_{a_{3}}t_{s}t_{x}.
\end{eqnarray*}
\begin{figure}[htbp]
 \begin{center}
  \includegraphics*[width=6cm]{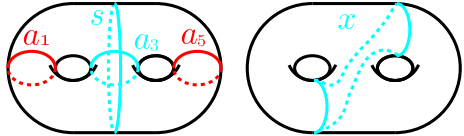}
 \end{center}
 \caption{The curves $s$, $a$, $b$, $c$, $x$, $y$ and $z$ on $\Sigma_{2}$.}
 \label{fig3}
\end{figure}
\begin{figure}[htbp]
 \begin{center}
  \includegraphics*[width=2.5cm]{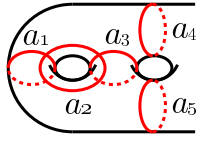}
 \end{center}
\caption{The curves $a_{1}a_{2}$, $a_{3}$, $a_{4}$ and $a_{5}$.}
 \label{fig2}
\end{figure}

The second relation is the \textit{chain relation}.
Let $a_{1}$, $a_{2}$, $a_{3}$, $a_{4}$ and $a_{5}$ be simple closed curves on $\Sigma_{g}$ in Figure~\ref{fig2}.
$t_{a_{1}}$, $t_{a_{2}}$, $t_{a_{3}}$, $t_{a_{4}}$ and $t_{a_{5}}$ satisfy the \textit{chain relation}
\begin{eqnarray*}
(t_{a_{3}}t_{a_{2}}t_{a_{1}})^{4}=t_{a_{4}}t_{a_{5}}.
\end{eqnarray*}
We note that if $g=2$, then $a_{4}$ is equal to $a_{5}$.
Therefore, in the case of $g=2$, the chain relation is as follows:
\begin{eqnarray*}
(t_{a_{3}}t_{a_{2}}t_{a_{1}})^{4}=t_{a_{5}}^{2}.
\end{eqnarray*}

\section{Proofs of the main results}
We prove Theorem~\ref{thm1}.
\begin{proof}[Proof of Theorem~\ref{thm1}]
Let $c$ and $s$ be a nonseparating curve and a separating curve on $\Sigma_{g}$, respectively.
Let $\varphi$ be a homogeneous quasimorphism on ${\mathcal M}_{g}$.

We first prove that ${\rm scl}_{{\mathcal M}_{g}}(t_{c})\leq \frac{1}{10}$ for $g\geq 2$.
Suppose that $g\geq 3$.
By the chain relation and Lemma~\ref{braid}
\begin{eqnarray*}
t_{a_{4}}t_{a_{5}}&=&t_{a_{3}}t_{a_{2}}\underline{t_{a_{1}}t_{a_{3}}}t_{a_{2}}t_{a_{1}}t_{a_{3}}t_{a_{2}}\underline{t_{a_{1}}t_{a_{3}}}t_{a_{2}}t_{a_{1}}\\
&=&t_{a_{3}}t_{a_{2}}t_{a_{3}}\underline{t_{a_{1}}t_{a_{2}}t_{a_{1}}}\underline{t_{a_{3}}t_{a_{2}}t_{a_{3}}}t_{a_{1}}t_{a_{2}}t_{a_{1}}\\
&=&t_{a_{3}}\underline{t_{a_{2}}t_{a_{3}}t_{a_{2}}}t_{a_{1}}t_{a_{2}}t_{a_{2}}t_{a_{3}}\underline{t_{a_{2}}t_{a_{1}}t_{a_{2}}}t_{a_{1}}\\
&=&t_{a_{3}}t_{a_{3}}t_{a_{2}}t_{a_{3}}t_{a_{1}}t_{a_{2}}t_{a_{2}}t_{a_{3}}t_{a_{1}}t_{a_{2}}t_{a_{1}}t_{a_{1}}\\
&=&t_{a_{3}}^{2}(t_{a_{2}}t_{a_{3}}t_{a_{1}}t_{a_{2}})(t_{a_{2}}t_{a_{3}}t_{a_{1}}t_{a_{2}})t_{a_{1}}^{2}.
\end{eqnarray*}
Therefore, we have $t_{a_{4}}t_{a_{5}}t_{a_{3}}^{-2}t_{a_{1}}^{-2}=(t_{a_{2}}t_{a_{3}}t_{a_{1}}t_{a_{2}})^{2}$.
Since $a_{1}$, $a_{3}$, $a_{4}$ and $a_{5}$ are disjoint from each other, from the definition of the homogeneous quasimorphism, Lemma~\ref{quasi1}, Lemma~\ref{conj} and Lemma~\ref{braid} we have
\begin{eqnarray*}
\varphi(t_{a_{2}}t_{a_{3}}t_{a_{1}}t_{a_{2}})&=&\frac{1}{2}(\varphi(t_{a_{4}}t_{a_{5}}t_{a_{3}}^{-2}t_{a_{1}}^{-2}))\\
&=&\frac{1}{2}(\varphi(t_{a_{4}})+\varphi(t_{a_{5}})-2\varphi(t_{a_{3}})-2\varphi(t_{a_{1}}))\\
&=&-\varphi(t_{c}).
\end{eqnarray*}
On the other hand, by Lemma~\ref{quasi1}
\begin{eqnarray*}
-\varphi(t_{c})=\varphi(t_{a_{2}}t_{a_{3}}t_{a_{1}}t_{a_{2}})=\varphi(t_{a_{2}}(t_{a_{2}}t_{a_{3}}t_{a_{1}}t_{a_{2}})t_{a_{2}}^{-1})=\varphi(t_{a_{2}}^{2}t_{a_{3}}t_{a_{1}}).
\end{eqnarray*}
Therefore, from the definition of the quasimorphim and Lemma~\ref{quasi1}
\begin{eqnarray*}
D(\varphi)&\geq&|\varphi(t_{a_{2}}^{2}t_{a_{3}}t_{a_{1}})-\varphi(t_{a_{2}}^{2})-\varphi(t_{a_{3}}t_{a_{1}})|\\
&=&|-\varphi(t_{c})-2\varphi(t_{a_{2}})-\varphi(t_{a_{3}})-\varphi(t_{a_{1}})|\\
&=&5|\varphi(t_{c})|.
\end{eqnarray*}
Since $\frac{|\varphi(t_{c})|}{2D(\varphi)}\leq \frac{1}{10}$, by Bavard's Duality Theorem we have ${\rm scl}_{{\mathcal M}_{g}}(t_{c})\leq \frac{1}{10}$.

If $g=2$, we use the chain relation $(t_{a_{3}}t_{a_{2}}t_{a_{1}})^{4}=t_{a_{5}}^{2}$.
By a similar argument we can prove that ${\rm scl}_{{\mathcal M}_{2}}(t_{c})\leq \frac{1}{10}$.

\

We next prove that ${\rm scl}_{{\mathcal M}_{g}}(t_{s})\leq \frac{1}{2}$ for $g\geq 3$.
By the lantern relation and Lemma~\ref{braid}
\begin{eqnarray*}
t_{s}t_{a}t_{b}t_{c}t_{z}^{-1}=t_{x}t_{y}.
\end{eqnarray*}
Hence, by Lemma~\ref{braid} and Lemma~\ref{quasi1} we have
\begin{eqnarray*}
\varphi(t_{x}t_{y})&=&\varphi(t_{s}t_{a}t_{b}t_{c}t_{z}^{-1})\\
&=&\varphi(t_{s})+\varphi(t_{a})+\varphi(t_{b})+\varphi(t_{c})-\varphi(t_{z})\\
&=&\varphi(t_{s})+2\varphi(t_{c}).
\end{eqnarray*}
From the definition of quasimorphism and Lemma~\ref{quasi1}
\begin{eqnarray*}
D(\varphi)&\geq&|\varphi(t_{x}t_{y})-\varphi(t_{x})-\varphi(t_{y})|\\
&=&|\varphi(t_{s})+2\varphi(t_{c})-\varphi(t_{x})-\varphi(t_{y})|\\
&=&|\varphi(t_{s})|.
\end{eqnarray*}
By Bavard's Duality Theorem we have ${\rm scl}_{{\mathcal M}_{g}}(t_{s})\leq \frac{1}{2}$.

\

Finally, we prove that ${\rm scl}_{{\mathcal M}_{2}}(t_{s})\leq 7/10$.
By the lantern relation we have $t_{a_{1}}^{2}t_{a_{5}}^{2}t_{a_{3}}^{-1}=t_{s}t_{x}$.
From the definition of quasimorphism and Lemma~\ref{quasi1}
\begin{eqnarray*}
D(\varphi)&\geq&|\varphi(t_{s}t_{x})-\varphi(t_{s})-\varphi(t_{x})|\\
&=&|\varphi(t_{a_{1}}^{2}t_{a_{5}}^{2}t_{a_{3}}^{-1})-\varphi(t_{s})-\varphi(t_{x})|\\
&=&|2\varphi(t_{a_{1}})+2\varphi(t_{a_{5}})-\varphi(t_{a_{3}})-\varphi(t_{s})-\varphi(t_{x})|\\
&=&|2\varphi(t_{c})-\varphi(t_{s})|\\
&\geq&|\varphi(t_{s})|-2|\varphi(t_{c})|.
\end{eqnarray*}
Therefore, we have $\frac{|\varphi(t_{s})|}{2D(\varphi)}\leq 2\frac{|\varphi(t_{c})|}{2D(\varphi)}+\frac{1}{2}$.
When we use Bavard's duality theorem for the left side after having used Bavard's duality theorem for the right side, we have ${\rm scl}_{{\mathcal M}_{2}}(t_{s})\leq 2{\rm scl}_{{\mathcal M}_{2}}(t_{c})+\frac{1}{2}$.
Since ${\rm scl}_{{\mathcal M}_{2}}(t_{c})\leq \frac{1}{10}$, ${\rm scl}_{{\mathcal M}_{2}}(t_{s})\leq \frac{7}{10}$.

This completes the proof of Theorem~\ref{thm1}.
\end{proof}

Let $\Sigma_{k}$ be a closed connected oriented surface of genus $k\geq 1$.
We denote the signature of a $4$-manifold $M$ as $\sigma(M)$ in the rest of this paper.
\begin{theorem}[{\rm [\cite{10}]},{\rm [\cite{11}]},{\rm [\cite{5}]}]\label{sign}
Let $M$ be a 4-manifold which admits a hyperelliptic Lefschetz fibration of genus $g$ over $\Sigma_{k}$.
Let $n$ and $s=\Sigma_{h=1}^{\frac{[g]}{2}}b_{h}$ be the numbers of nonseparating and separating vanishing cycles in the global monodromy of this fibration, respectively.
Then
\begin{eqnarray*}
\sigma(M)=-\frac{g+1}{2g+1}n+\Sigma_{h=1}^{[\frac{g}{2}]}(\frac{4h(g-h)}{2g+1}-1)b_{h},
\end{eqnarray*}
where $b_{h}$ denotes the number of separating vanishing cycles which separate the genus $g$ surface into two surfaces one of which has genus $h$.
\end{theorem}
We prove Theorem~\ref{thm2}.
\begin{proof}[Proof of Theorem~\ref{thm2}]
We base on the argument of [\cite{9}].
Suppose that $g\geq 3$.
We assume the contrary that ${\rm scl}_{{\mathcal H}_{g}}(t_{c})<\frac{1}{4(2g+1)}$.
Choose a rational number $r$ with ${\rm scl}_{{\mathcal H}_{g}}(t_{c})<r<\frac{1}{4(2g+1)}$.
Then there exists an arbitrarily large positive integer $n$ such that $rn$ is an integer and $t_{c}^{n}$ can be written as a product of $rn$ commutators in ${\mathcal H}_{g}$.
Note that we take $n$ as a multiple of $4(2g+1)$.
This gives a relatively minimal genus-$g$ Lefschetz fibration $M\rightarrow \Sigma_{rn}$ with the nonseparating vanishing cycle $c$ repeated $n$ times.
(More details of the theory of Lefschetz fibrations can be found in [\cite{15}]).

In [\cite{9}], Korkmaz gave an upper bound for the signature of $M$;
\begin{eqnarray*}
\sigma(M)\leq 4grn-n+4.
\end{eqnarray*}
On the other hand, by Theorem~\ref{sign} we see
\begin{eqnarray*}
\sigma(M)=-\frac{g+1}{2g+1}n.
\end{eqnarray*}
Hence, we obtain
\begin{eqnarray*}
-\frac{g+1}{2g+1}n\leq 4grn-n+4.
\end{eqnarray*}
As a result of this, we conclude that there exists arbitrarily big $n$ such that
\begin{eqnarray*}
0\leq (4r-\frac{1}{2g+1})gn+4.
\end{eqnarray*}
Since $r-\frac{1}{4(2g+1)}$ is negative, this contradicts to the inequality.
This completes the proof of $\frac{1}{4(2g+1)}\leq{\rm scl}_{{\mathcal H}_{g}}(t_{c})$.
%
%
By a similar argument we can prove that $\frac{h(g-h)}{g(2g+1)}\leq{\rm scl}_{{\mathcal H}_{g}}(t_{s_{h}})$ \ $(h=1,\ldots,\frac{[g]}{2})$ for $g\geq 3$.
%
%
This completes the proof of Theorem~\ref{thm2}.
\end{proof}

Finally, we will show Lemma~\ref{10}.
\begin{proof}[Proof of Lemma~\ref{10}]
Let $s$, $a_{1}$ and $a_{2}$ be simple closed curves in Figure~\ref{fig3}.
It is well known that $t_{s}$, $t_{a_{1}}$ and $t_{a_{2}}$ satisfy the following relation.
\begin{eqnarray*}
t_{s}=(t_{a_{2}}t_{a_{1}})^{6}.
\end{eqnarray*}
By Lemma~\ref{braid} 
\begin{eqnarray*}
t_{s}&=&\underline{t_{a_{2}}t_{a_{1}}t_{a_{2}}t_{a_{1}}t_{a_{2}}t_{a_{1}}}t_{a_{2}}t_{a_{1}}\underline{t_{a_{2}}t_{a_{1}}t_{a_{2}}}t_{a_{1}}\\
&=&t_{a_{1}}\underline{t_{a_{2}}t_{a_{1}}t_{a_{2}}}t_{a_{1}}t_{a_{2}}t_{a_{2}}t_{a_{1}}t_{a_{1}}t_{a_{2}}t_{a_{1}}}t_{a_{1}\\
&=&t_{a_{1}}t_{a_{1}}t_{a_{2}}t_{a_{1}}t_{a_{1}}t_{a_{2}}t_{a_{2}}t_{a_{1}}t_{a_{1}}t_{a_{2}}t_{a_{1}}}t_{a_{1}
\end{eqnarray*}
Therefore, from Lemma~\ref{braid} we have $t_{s}t_{a_{1}}^{-4}=(t_{a_{2}}t_{a_{1}}t_{a_{1}}t_{a_{2}})^{2}$.

Let $\varphi$ be a homogeneous quasimorphism on ${\mathcal M}_{2}$.
By the definition of homogeneous quasimorphisms and Lemma~\ref{quasi1} we have
\begin{eqnarray*}
\frac{1}{2}\varphi(t_{s})-2\varphi(t_{a_{1}})=\varphi(t_{a_{2}}t_{a_{1}}t_{a_{1}}t_{a_{2}})=\varphi(t_{a_{2}}(t_{a_{2}}t_{a_{1}}t_{a_{1}}t_{a_{2}})t_{a_{2}}^{-1})=\varphi(t_{a_{2}}^{2}t_{a_{1}}^{2}).
\end{eqnarray*}
Let $c$ be a nonseparating curve on $\Sigma_{2}$.
From the definition of quasimorphisms and homogeneous quasimorphisms
\begin{eqnarray*}
D(\varphi)&\geq&|\varphi(t_{a_{2}}^{2}t_{a_{1}}^{2})-\varphi(t_{a_{2}}^{2})-\varphi(t_{a_{1}}^{2})|\\
&=&|\varphi(t_{a_{2}}^{2}t_{a_{1}}^{2})-2\varphi(t_{a_{2}})-2\varphi(t_{a_{1}})|\\
&=&|\frac{1}{2}\varphi(t_{s})-2\varphi(t_{a_{1}})-2\varphi(t_{a_{2}})-2\varphi(t_{a_{1}})|\\
&=&|\frac{1}{2}\varphi(t_{s})-6\varphi(t_{c})|\\
&\geq&6|\varphi(t_{c})|-\frac{1}{2}|\varphi(t_{s})|.
\end{eqnarray*}
Therefore, we have $6\frac{|\varphi(t_{c})|}{2D(\varphi)}\leq \frac{1}{2}\frac{|\varphi(t_{s})|}{2D(\varphi)}+\frac{1}{2}$.
When we use Bavard's duality theorem for the left side after having used Bavard's duality theorem for the right side, we have
\begin{eqnarray*}
6{\rm scl}_{{\mathcal M}_{2}}(t_{c})\leq \frac{1}{2}{\rm scl}_{{\mathcal M}_{2}}(t_{s})+\frac{1}{2}.
\end{eqnarray*}
This completes the proof of Lemma~\ref{10}.
\end{proof}
\section*{Acknowledgment}
The author would like to thank Andrew Putman for pointing out the mistake of references in this paper.

\ \\
Department of Mathematics, Graduate School of Science, Osaka University, Toyonaka, Osaka 560-0043, Japan\\
\textit{E-mail address:} \bf{n-monden@cr.math.sci.osaka-u.ac.jp}
\end{document}